\documentclass[reqno]{amsart}

\usepackage{amssymb}
\usepackage{graphicx}
\usepackage{amscd}
\usepackage[pagebackref]{hyperref}
\usepackage{color}
\usepackage{tabularx}
\usepackage[table]{xcolor}
\usepackage{float}
\usepackage{graphics,amsmath,amssymb}
\usepackage{amsthm}
\usepackage{amsfonts}
\usepackage{latexsym}
\usepackage{epsf}
\usepackage{xifthen}
\usepackage{mathrsfs}
\usepackage{dsfont}
\usepackage{makecell}
\usepackage{subfig}
\usepackage{amsmath}
\allowdisplaybreaks[4]
\usepackage{listings}
\usepackage{etoolbox}
\usepackage{fancyhdr}
\usepackage{pdflscape}
\usepackage[title,toc,titletoc]{appendix}
\usepackage{enumitem}
\usepackage[noadjust]{cite}
\usepackage{tikz}
\usetikzlibrary{automata,positioning,arrows}
\usepackage{young}
\usepackage[object=vectorian]{pgfornament} 
\usepackage{lipsum,tikz}
\usepackage{multirow}
\usepackage[OT2,T1]{fontenc}
\usepackage{mathtools}
\usepackage{ytableau}

\hypersetup{
	colorlinks=true, 
	linktoc=all, 
	linkcolor=blue} 

\numberwithin{equation}{section}

\theoremstyle{theorem}
\newtheorem{theorem}{Theorem}[section]
\newtheorem*{theorem*}{Theorem}

\newtheorem{corollary}[theorem]{Corollary}
\newtheorem{lemma}[theorem]{Lemma}

\newtheorem{question}[theorem]{Question}

\newtheorem{innercustomgeneric}{\customgenericname}
\providecommand{\customgenericname}{}
\newcommand{\newcustomtheorem}[2]{%
	\newenvironment{#1}[1]
	{%
		\renewcommand\customgenericname{#2}%
		\renewcommand\theinnercustomgeneric{##1}%
		\innercustomgeneric
	}
	{\endinnercustomgeneric}
}
\newcustomtheorem{ctheorem}{Theorem}
\newcustomtheorem{clemma}{Lemma}

\theoremstyle{definition}
\newtheorem{definition}[theorem]{Definition}

\newtheorem*{example*}{Example}
\newtheorem*{examples*}{Examples}
\newtheorem{remark}[theorem]{Remark}
\newtheorem*{remark*}{Remark}
\newtheorem*{remarks*}{Remarks}
\newtheorem*{note*}{Note}

\newtheoremstyle{named}{}{}{\itshape}{}{\bfseries}{.}{.5em}{#1\thmnote{ #3}}
\theoremstyle{named}

\DeclareMathAlphabet{\mydutchcal}{U}{dutchcal}{m}{n}
\newcommand{\dcD}{\mydutchcal{D}}
\newcommand{\dcK}{\mydutchcal{K}}

\newcommand{\qbinom}[2]{{#1\brack #2}}
\newcommand{\tqbinom}[2]{{\textstyle{#1\brack #2}}}

\newcommand{\RHS}{\operatorname{RHS}}

\newcommand{\Res}{\operatorname{Res}}
\newcommand{\parity}{\operatorname{par}}

\newcommand{\lrg}{\mathsf{lrg}}
\newcommand{\len}{\mathsf{len}}

\newcommand{\rT}{\mathrm{T}}
\newcommand{\Tfnc}[4]{\mathrm{T}_{#1}\!\left(\begin{matrix}
		#2\\#3
	\end{matrix}\,;#4\right)}

\title[Finite Kleshchev bipartitions]{Finite Kleshchev bipartitions and $q$-trinomial coefficients}

\author[S. Chern]{Shane Chern}
\address{Fakult\"at f\"ur Mathematik, Universit\"at Wien, Oskar-Morgenstern-Platz 1, Wien 1090, Austria}
\email{chenxiaohang92@gmail.com, xiaohangc92@univie.ac.at}

\date{}



\begin{document}
	
\sloppy

\begin{abstract}
	The Kleshchev multipartitions arise in the representation theory for the Ariki--Koike algebras. In previous work, Li, Stanton, Xue, Yee, and the author considered a refined enumeration for the $2$-dimensional case, namely, the Kleshchev bipartitions, by invoking the $2$-residue statistic for partitions. In this paper, we make further elaboration by bounding the largest part of the bipartitions and show that the related counting functions are connected with two families of $q$-trinomial coefficients introduced by Andrews and Baxter.
\end{abstract}

\maketitle

\section{Introduction}

A \emph{partition} $\lambda$ of a natural number $n$ is a nondecreasing sequence of positive integers, known as \emph{parts} in $\lambda$, such that their sum equals $n$. In particular, we allow the sequence to be empty, giving the unique partition of $0$, which is referred to as the \emph{empty partition}. Let us call $n$ the \emph{size} of $\lambda$, denoted by $|\lambda|$. Also, the \emph{length} of $\lambda$ is defined as the number of parts in this partition. We introduce the statistics 
\begin{align*}
	\len(\lambda) &:= \text{length of $\lambda$},\\
	\lrg(\lambda) &:= \text{largest part in $\lambda$}.
\end{align*}
In this work, we mainly focus on partitions into distinct parts, known as \emph{strict partitions}.

Exploring the representations of the Ariki--Koike algebras, which were independently introduced in \cite{AK1994} and \cite{BM1993}, Ariki and Mathas~\cite{Ari1996, AM2000} observed that the simple modules of these algebras can be labeled by the \emph{Kleshchev multipartitions}, which are certain tuples of strict partitions. As such, it was further shown in \cite{AM2000} that the generating function for Kleshchev multipartitions can be derived Lie-algebraically by combining Ariki's categorification theorem and the Weyl--Kac character formula. Along a different road, Li, Stanton, Xue, Yee, and the author~\cite{CLSXY2024} revisited this generating function combinatorially and established a new multiple Rogers--Ramanujan type identity by basic hypergeometric means.

In the second part of \cite{CLSXY2024}, our emphasis was placed on the $2$-dimensional case, namely, the \emph{Kleshchev bipartitions}. Here, we define such bipartitions in the way of Mathas~\cite{Mat1998}, or more precisely, according to \cite[pp.~497--498, Definition~2.3 and Remark~2.1]{CLSXY2024}.

\begin{definition}
	For $a\in \{1,2\}$, the \emph{Kleshchev bipartition} set $\Lambda^{a,2}$ consists of pairs $(\pi^{(1)}, \pi^{(2)})$ of \emph{strict} partitions such that
	\begin{align*}
		\lrg(\pi^{(1)}) + t_1 \le \len(\pi^{(2)}) + t_2,
	\end{align*}
	where
	\begin{align*}
		t_i = \begin{cases}
			0, & \text{if $1\le i\le a$},\\
			1, & \text{if $a+1\le i\le 2$}.
		\end{cases}
	\end{align*}
\end{definition}

Specifically, we considered in \cite{CLSXY2024} a refined enumeration of Kleshchev bipartitions by invoking the $2$-residue statistic for partitions. Recall that given a partition $\lambda := (\lambda_1,\lambda_2,\ldots,\lambda_l)$, its \emph{Young diagram} depicts this partition by boxes placed in rows such that there are $\lambda_i$ boxes in the $i$-th row. For the box in the $i$-th row and $j$-th column, we define its \emph{$2$-residue} as
\begin{align*}
	\Res_\lambda(i,j) := (i-j) \bmod{2}.
\end{align*}
For example, the Young diagram of the partition $(5,3,2,2)$ and the $2$-residue of each box are as follows:
\begin{align*}
	\vcenter{\vspace*{5pt}\hbox{\begin{ytableau}
				0 & 1 & 0 & 1 & 0\\
				1 & 0 & 1\\
				0 & 1\\
				1 & 0
	\end{ytableau}}\vspace*{5pt}}\quad.
\end{align*}
Now define the \emph{$2$-residue difference} statistic for $\lambda$ as
\begin{align*}
	\omega(\lambda) := \#\big\{(i,j) : \Res_\lambda(i,j) = 0\big\} - \#\big\{(i,j) : \Res_\lambda(i,j) = 1\big\}.
\end{align*}
This statistic, as indicated in \cite{BG2008}, is in essence the \emph{BG-rank} introduced by Berkovich and Garvan~\cite{BG2006}.

For any Kleshchev bipartition $\pi = (\pi^{(1)},\pi^{(2)}) \in \Lambda^{a,2}$, let us abuse the notation
\begin{align*}
	|\pi| := |\pi^{(1)}| + |\pi^{(2)}|,
\end{align*}
and
\begin{align*}
	\lrg(\pi) := \lrg(\pi^{(2)}).
\end{align*}
In addition, we write
\begin{align*}
	\omega(\pi) = \omega^{a,2}(\pi) := (-1)^{t_1}\omega(\pi^{(1)}) + (-1)^{t_2}\omega(\pi^{(2)}).
\end{align*}
The following enumerations were established in \cite[p.~494, Theorem~1.3 and Corollary~1.4]{CLSXY2024}.

\begin{theorem}[Chern--Li--Stanton--Xue--Yee]\label{th:CLSXY}
	For $s\in \mathbb{Z}$,
	\begin{align}\label{eq:CLSXY-2}
		\sum_{\substack{\pi \in \Lambda^{2,2}\\ \omega(\pi) = s}} q^{|\pi|} = q^{s(s-1)} \frac{(-q;q^2)_\infty+ (-1)^s(q;q^2)_\infty}{2(q^2;q^2)_\infty},
	\end{align}
	and
	\begin{align}\label{eq:CLSXY-1}
		\sum_{\substack{\pi \in \Lambda^{1,2}\\ \omega(\pi) = s}} q^{|\pi|} = q^{s^2} \frac{(-q^2;q^2)_\infty}{(q^2;q^2)_\infty}.
	\end{align}
\end{theorem}

Throughout, we adopt the conventional \emph{$q$-Pochhammer symbols} for $n\in \mathbb{Z}$:
\begin{align*}
	(a;q)_\infty := \prod_{k\ge 0} (1-aq^k),\qquad\qquad (a;q)_n := \frac{(a;q)_\infty}{(aq^n;q)_\infty}.
\end{align*}
The infinite products in \eqref{eq:CLSXY-2} and \eqref{eq:CLSXY-1} easily remind us of the limiting case of the \emph{$q$-trinomial coefficients} introduced by Andrews and Baxter~\cite{AB1987}. In particular, the following two versions \cite[p.~299, eqs.~(2.8) and (2.9)]{AB1987} are relevant:
\begin{align}
	\rT_0(L,A;q)=\Tfnc{0}{L}{A}{q} &:= \sum_{l=0}^L (-1)^l \qbinom{L}{l}_{q^2} \qbinom{2L-2l}{L-A-l}_q,\label{eq:T0}\\
	\rT_1(L,A;q)=\Tfnc{1}{L}{A}{q} &:= \sum_{l=0}^L (-q)^l \qbinom{L}{l}_{q^2} \qbinom{2L-2l}{L-A-l}_q,\label{eq:T1}
\end{align}
where the \emph{$q$-binomial coefficients} are defined for $M,N\in \mathbb{Z}$ by
\begin{align*}
	\qbinom{N}{M}_q := \frac{(q^{N-M+1};q)_M}{(q;q)_M}.
\end{align*}
It was shown in \cite[p.~310, eqs.~(2.53) and (2.54)]{AB1987} that
\begin{align*}
	\lim_{L\to \infty} \Tfnc{0}{L}{A}{q} = \begin{cases}
		\dfrac{(-q;q^2)_\infty+(q;q^2)_\infty}{2(q^2;q^2)_\infty}, & \text{if $L-A$ is even},\\[15pt]
		\dfrac{(-q;q^2)_\infty-(q;q^2)_\infty}{2(q^2;q^2)_\infty}, & \text{if $L-A$ is odd},
	\end{cases}
\end{align*}
and in \cite[p.~310, eq.~(2.51)]{AB1987} that
\begin{align*}
	\lim_{L\to \infty} \Tfnc{1}{L}{A}{q} = \frac{(-q^2;q^2)_\infty}{(q^2;q^2)_\infty}.
\end{align*}

As such, it is a natural question to ask if Kleshchev bipartitions are coincidentally connected with the two limits. To reveal such connections, we are motivated to consider \emph{finite} Kleshchev bipartitions and prove the following relations.

\begin{theorem}
	For $N\in \mathbb{Z}_{\ge 0}$ and $s\in \mathbb{Z}$,
	\begin{align}\label{eq:GF-2}
		\sum_{\substack{\pi \in \Lambda^{2,2}\\\lrg(\pi)\le N\\ \omega(\pi) = s}} q^{|\pi|} = q^{s(s-1)}\Tfnc{0}{N}{s-N+2\lfloor\tfrac{N}{2}\rfloor}{q},
	\end{align}
	and
	\begin{align}\label{eq:GF-1}
		\sum_{\substack{\pi \in \Lambda^{1,2}\\\lrg(\pi)\le N\\ \omega(\pi) = s}} q^{|\pi|} = q^{s^2}\left[\Tfnc{1}{N}{s}{q} + q^{N-(-1)^N s + 1} \Tfnc{0}{N}{s-(-1)^N}{q}\right].
	\end{align}
\end{theorem}

\begin{remark}
	In view of \eqref{eq:GF-2}, each $q$-trinomial coefficient $\rT_0(N,s;q)$ counts a certain finite subclass of Kleshchev bipartitions in $\Lambda^{2,2}$, thereby indicating that it is a \emph{nonnegative polynomial} in $q$. This fact cannot be directly seen from the definition in \eqref{eq:T0} due to the alternating sign of the summands.
\end{remark}

\section{Functional equations for Kleshchev-type bipartitions}

To facilitate our analysis, we introduce three families of Kleshchev-type bipartitions in a slightly different narrative:

\begin{definition}
	For $i\in \{-1,0,1\}$, the set $\dcK^{(i)}$ consists of pairs $(\pi^{(1)}, \pi^{(2)})$ of \emph{strict} partitions such that
	\begin{align*}
		\lrg(\pi^{(1)}) \le \len(\pi^{(2)}) + i.
	\end{align*}
	We further denote by $\dcK_N^{(i)}$ the subset of $\dcK^{(i)}$ such that $\lrg(\pi^{(2)}) \le N$.
\end{definition}

\begin{remark}
	It is clear that $\dcK^{(0)}$ is exactly $\Lambda^{2,2}$, while $\dcK^{(1)}$ is $\Lambda^{1,2}$.
\end{remark}

Now let us define the generating functions
\begin{align*}
	G_N^{(i)}(x) = G_N^{(i)}(x,q) := \sum_{(\pi^{(1)}, \pi^{(2)}) \in \dcK_N^{(i)}} x^{\omega(\pi^{(1)})-\omega(\pi^{(2)})} q^{|\pi^{(1)}|+|\pi^{(2)}|}, \qquad (i= \pm 1).
\end{align*}
In addition, we define
\begin{align*}
	G_N^{(0), \pm}(x) = G_N^{(0), \pm}(x,q) := \sum_{(\pi^{(1)}, \pi^{(2)}) \in \dcK_N^{(0)}} x^{\omega(\pi^{(1)})\pm \omega(\pi^{(2)})} q^{|\pi^{(1)}|+|\pi^{(2)}|},
\end{align*}
and write
\begin{align*}
	G_N^{(0)}(x) = G_N^{(0)}(x,q) := G_N^{(0), +}(x,q).
\end{align*}

The objective of this section is to construct the following functional equations for the three generating functions.

\begin{lemma}\label{le:fnc}
	For $N\ge 1$,
	\begin{align}\label{eq:fnc-1}
		G_N^{(0)}(x) = G_{N-1}^{(0)}(x) + x^{\parity(N)} q^N G_{N-1}^{(1)}(x),
	\end{align}
	and
	\begin{align}\label{eq:fnc-2}
		G_N^{(-1)}(x) &= G_{N-1}^{(-1)}(x) + x^{-\parity(N)} q^N G_{N-1}^{(0)}(x),
	\end{align}
	where the parity function is given by
	\begin{align*}
		\parity(N) := \begin{cases}
			0, & \text{if $N$ is even},\\
			1, & \text{if $N$ is odd}.
		\end{cases}
	\end{align*}
	In addition,
	\begin{align}\label{eq:fnc-3}
		G_N^{(-1)}(x) = G_{N-1}^{(1)}(x^{-1}) - G_{N-1}^{(-1)}(x^{-1}) - 1.
	\end{align}
\end{lemma}

\begin{proof}
	The functional equations \eqref{eq:fnc-1} and \eqref{eq:fnc-2} are relatively easy. We start with a bipartition $(\pi^{(1)}, \pi^{(2)}) \in \dcK_N^{(-1)}$. If $\lrg(\pi^{(2)}) \le N-1$, then $(\pi^{(1)}, \pi^{(2)})$ also belongs to $\dcK_{N-1}^{(-1)}$, thereby giving the first term $G_{N-1}^{(-1)}(x)$ on the right-hand side of \eqref{eq:fnc-2}. If $\lrg(\pi^{(2)}) = N$, we preserve $\pi^{(1)}$ and remove the largest part $N$ in $\pi^{(2)}$ to get a new bipartition $(\widetilde{\pi}^{(1)}, \widetilde{\pi}^{(2)})$. In particular,
	\begin{align*}
		\lrg(\widetilde{\pi}^{(1)}) = \lrg(\pi^{(1)}) \le \len(\pi^{(2)})-1 = \len(\widetilde{\pi}^{(2)}),
	\end{align*}
	so that $(\widetilde{\pi}^{(1)}, \widetilde{\pi}^{(2)})\in \dcK_{N-1}^{(0)}$. Moreover,
	\begin{align*}
		\omega(\pi^{(1)})- \omega(\pi^{(2)}) = \omega(\widetilde{\pi}^{(1)})+ \omega(\widetilde{\pi}^{(2)}) - \parity(N).
	\end{align*}
	Hence, the second term $x^{-\parity(N)} q^N G_{N-1}^{(0)}(x)$ on the right-hand side of \eqref{eq:fnc-2} is derived. For \eqref{eq:fnc-1}, we apply a similar argument.
	
	Next we prove \eqref{eq:fnc-3}, which is the most intricate. Let $(\pi^{(1)}, \pi^{(2)}) \in \dcK_N^{(-1)}$ be such that $\len(\pi^{(2)}) = M \ge 1$. Then $\lrg(\pi^{(1)})\le M-1$. Now construct a new bipartition $(\widehat{\pi}^{(1)}, \widehat{\pi}^{(2)})$, where $\widehat{\pi}^{(1)}$ is obtained by prepending $M$ to $\pi^{(1)}$ as the largest part, and $\widehat{\pi}^{(2)}$ is obtained by deleting $1$ from each part of $\pi^{(2)}$. Thus,
	\begin{align*}
		\lrg(\widehat{\pi}^{(2)}) = \lrg(\pi^{(2)}) - 1 \le N-1.
	\end{align*}
	In addition,
	\begin{align*}
		|\pi^{(1)}|+|\pi^{(2)}| = |\widehat{\pi}^{(1)}|+|\widehat{\pi}^{(2)}|,
	\end{align*}
	and
	\begin{align*}
		\omega(\pi^{(1)}) - \omega(\pi^{(2)}) = -\omega(\widehat{\pi}^{(1)}) + \omega(\widehat{\pi}^{(2)}).
	\end{align*}
	For the moment, we have two cases. If $1$ is a part in $\pi^{(2)}$, then
	\begin{align*}
		\lrg(\widehat{\pi}^{(1)}) = M = \len(\pi^{(2)}) = \len(\widehat{\pi}^{(2)}) + 1.
	\end{align*}
	Now,
	\begin{align*}
		&\sum_{M\ge 1} \sum_{\substack{(\widehat{\pi}^{(1)}, \widehat{\pi}^{(2)})\\\lrg(\widehat{\pi}^{(2)})\le N-1 \\ \lrg(\widehat{\pi}^{(1)}) = M\\ \len(\widehat{\pi}^{(2)}) = M - 1}} x^{-\omega(\widehat{\pi}^{(1)})+\omega(\widehat{\pi}^{(2)})} q^{|\widehat{\pi}^{(1)}|+|\widehat{\pi}^{(2)}|}\\
		&\qquad = \sum_{M\ge 1} \left(\sum_{\substack{(\widehat{\pi}^{(1)}, \widehat{\pi}^{(2)})\in \dcK_{N-1}^{(1)}\\ \len(\widehat{\pi}^{(2)}) = M - 1}} - \sum_{\substack{(\widehat{\pi}^{(1)}, \widehat{\pi}^{(2)})\in \dcK_{N-1}^{(0)}\\ \len(\widehat{\pi}^{(2)}) = M - 1}}\right) x^{-\omega(\widehat{\pi}^{(1)})+\omega(\widehat{\pi}^{(2)})} q^{|\widehat{\pi}^{(1)}|+|\widehat{\pi}^{(2)}|}\\
		&\qquad = G_{N-1}^{(1)}(x^{-1}) - G_{N-1}^{(0),-}(x^{-1}).
	\end{align*}
	If $1$ is not a part in $\pi^{(2)}$, then
	\begin{align*}
		\lrg(\widehat{\pi}^{(1)}) = M = \len(\pi^{(2)}) = \len(\widehat{\pi}^{(2)}).
	\end{align*}
	Now,
	\begin{align*}
		&\sum_{M\ge 1} \sum_{\substack{(\widehat{\pi}^{(1)}, \widehat{\pi}^{(2)})\\\lrg(\widehat{\pi}^{(2)})\le N-1\\ \lrg(\widehat{\pi}^{(1)}) = M\\ \len(\widehat{\pi}^{(2)}) = M}} x^{-\omega(\widehat{\pi}^{(1)})+\omega(\widehat{\pi}^{(2)})} q^{|\widehat{\pi}^{(1)}|+|\widehat{\pi}^{(2)}|}\\
		&\qquad = \sum_{M\ge 1} \left(\sum_{\substack{(\widehat{\pi}^{(1)}, \widehat{\pi}^{(2)})\in \dcK_{N-1}^{(0)}\\ \len(\widehat{\pi}^{(2)}) = M}} - \sum_{\substack{(\widehat{\pi}^{(1)}, \widehat{\pi}^{(2)})\in \dcK_{N-1}^{(-1)}\\ \len(\widehat{\pi}^{(2)}) = M}}\right) x^{-\omega(\widehat{\pi}^{(1)})+\omega(\widehat{\pi}^{(2)})} q^{|\widehat{\pi}^{(1)}|+|\widehat{\pi}^{(2)}|}\\
		&\qquad = G_{N-1}^{(0),-}(x^{-1}) - 1 - G_{N-1}^{(-1)}(x^{-1}).
	\end{align*}
	It follows that
	\begin{align*}
		G_N^{(-1)}(x) = \left(G_{N-1}^{(1)}(x^{-1}) - G_{N-1}^{(0),-}(x^{-1})\right) + \left(G_{N-1}^{(0),-}(x^{-1}) - 1 - G_{N-1}^{(-1)}(x^{-1})\right),
	\end{align*}
	thereby producing the right-hand side of \eqref{eq:fnc-3}.
\end{proof}

With the functional equations in Lemma~\ref{le:fnc}, we may further derive a functional equation for $G_N^{(0)}(x)$ itself.

\begin{corollary}
	For $N\ge 2$,
	\begin{align}\label{eq:fnc-G0}
		G_N^{(0)}(x) &= G_{N-1}^{(0)}(x) + x^{2\parity(N)-1} q G_{N-1}^{(0)}(x) + x^{2\parity(N)} q^{2N} G_{N-1}^{(0)}(x^{-1})\notag\\
		&\quad + x^{2\parity(N)-1} q^{2N-1} G_{N-2}^{(0)}(x) - x^{2\parity(N)-1} q G_{N-2}^{(0)}(x).
	\end{align}
\end{corollary}

\begin{proof}
	We first rewrite \eqref{eq:fnc-2} as
	\begin{align}\label{eq:middle-a1}
		G_N^{(-1)}(x) - G_{N-1}^{(-1)}(x) = x^{-\parity(N)} q^N G_{N-1}^{(0)}(x).
	\end{align}
	Meanwhile, in \eqref{eq:fnc-3}, we substitute $G^{(1)}$ in terms of $G^{(0)}$ using \eqref{eq:fnc-1} and obtain
	\begin{align}\label{eq:middle-a2}
		G_N^{(-1)}(x) + G_{N-1}^{(-1)}(x^{-1}) = x^{\parity(N)} q^{-N} G_{N}^{(0)}(x^{-1}) - x^{\parity(N)} q^{-N} G_{N-1}^{(0)}(x^{-1}) - 1.
	\end{align}
	Taking the difference of \eqref{eq:middle-a2} and \eqref{eq:middle-a1},
	\begin{align}\label{eq:middle-b1}
		G_{N-1}^{(-1)}(x) + G_{N-1}^{(-1)}(x^{-1}) &= x^{\parity(N)} q^{-N} G_{N}^{(0)}(x^{-1}) - x^{\parity(N)} q^{-N} G_{N-1}^{(0)}(x^{-1})\notag\\
		&\quad - x^{-\parity(N)} q^N G_{N-1}^{(0)}(x) - 1.
	\end{align}
	Also, replacing $x$ with $x^{-1}$ in \eqref{eq:middle-a1} and then summing with \eqref{eq:middle-a2},
	\begin{align}\label{eq:middle-b2}
		G_{N}^{(-1)}(x) + G_{N}^{(-1)}(x^{-1}) &= x^{\parity(N)} q^{-N} G_{N}^{(0)}(x^{-1}) - x^{\parity(N)} q^{-N} G_{N-1}^{(0)}(x^{-1})\notag\\
		&\quad + x^{\parity(N)} q^N G_{N-1}^{(0)}(x^{-1}) - 1.
	\end{align}
	Finally, in \eqref{eq:middle-b2}, we substitute $N$ by $N-1$, and then substract it from \eqref{eq:middle-b1}. Now replacing $x$ with $x^{-1}$, the desired functional equation \eqref{eq:fnc-G0} for $G_N^{(0)}(x)$ follows.
\end{proof}

\section{Proof of the main result}

\subsection{Kleshchev bipartition set $\Lambda^{2,2}$}

For the moment, let us prove \eqref{eq:GF-2}. We start with a direct verification that \eqref{eq:GF-2} holds for $N=0$ and $1$. Write
\begin{align*}
	G_N^{(0)}(x) = \sum_{s=-\infty}^\infty c_{s,N} x^s.
\end{align*}
Then
\begin{align*}
	\sum_{\substack{\pi \in \Lambda^{2,2}\\\lrg(\pi)\le N\\ \omega(\pi) = s}} q^{|\pi|} = c_{s,N}.
\end{align*}

In light of the functional equation \eqref{eq:fnc-G0}, we have a recurrence for $c_{s,N}$, separated into two cases according to the parity of $N$:
\begin{align}\label{eq:rec-even}
	c_{s,2n} = c_{s,2n-1} + q c_{s+1,2n-1} + q^{4n} c_{-s,2n-1} + q^{4n-1} c_{s+1,2n-2} - q c_{s+1,2n-2},
\end{align}
and
\begin{align}\label{eq:rec-odd}
	c_{s,2n+1} &= c_{s,2n} + q c_{s-1,2n} + q^{4n+2} c_{-s+2,2n} + q^{4n+1} c_{s-1,2n-1} - q c_{s-1,2n-1}.
\end{align}
To show
\begin{align*}
	c_{s,N} = q^{s(s-1)}\Tfnc{0}{N}{s-N+2\lfloor\tfrac{N}{2}\rfloor}{q},
\end{align*}
it is sufficient to verify that the right-hand side of the above also satisfies \eqref{eq:rec-even} and \eqref{eq:rec-odd}. For $N=2n$, we are supposed to show
\begin{align*}
	\Tfnc{0}{2n}{s}{q} &= \Tfnc{0}{2n-1}{s-1}{q} + q^{2s+1} \Tfnc{0}{2n-1}{s}{q} + q^{4n+2s} \Tfnc{0}{2n-1}{s+1}{q}\\
	&\quad+ q^{4n+2s-1} \Tfnc{0}{2n-2}{s+1}{q} - q^{2s+1} \Tfnc{0}{2n-2}{s+1}{q}.
\end{align*}
For $N=2n+1$, we are supposed to show
\begin{align*}
	\Tfnc{0}{2n+1}{s}{q} &= \Tfnc{0}{2n}{s+1}{q} + q^{-2s+1} \Tfnc{0}{2n}{s}{q} + q^{4n-2s+2} \Tfnc{0}{2n}{s-1}{q}\\
	&\quad+ q^{4n-2s+1} \Tfnc{0}{2n-1}{s-1}{q} - q^{-2s+1} \Tfnc{0}{2n-1}{s-1}{q}.
\end{align*}
Here we make the substitution $s\mapsto -s$, and note from \eqref{eq:T0} the symmetry
\begin{align*}
	\Tfnc{0}{N}{s}{q} = \Tfnc{0}{N}{-s}{q}.
\end{align*}
Then the previous relation is equivalent to
\begin{align*}
	\Tfnc{0}{2n+1}{s}{q} &= \Tfnc{0}{2n}{s-1}{q} + q^{2s+1} \Tfnc{0}{2n}{s}{q} + q^{4n+2s+2} \Tfnc{0}{2n}{s+1}{q}\\
	&\quad+ q^{4n+2s+1} \Tfnc{0}{2n-1}{s+1}{q} - q^{2s+1} \Tfnc{0}{2n-1}{s+1}{q}.
\end{align*}
Overall, we only need to prove 
\begin{align}\label{eq:rec-T0-to-show}
	\Tfnc{0}{N}{s}{q} &= \Tfnc{0}{N-1}{s-1}{q} + q^{2s+1} \Tfnc{0}{N-1}{s}{q} + q^{2N+2s} \Tfnc{0}{N-1}{s+1}{q}\notag\\
	&\quad+ q^{2N+2s-1} \Tfnc{0}{N-2}{s+1}{q} - q^{2s+1} \Tfnc{0}{N-2}{s+1}{q}.
\end{align}

Let us recall two recurrences for the $q$-trinomial coefficients $\rT_0(N,s;q)$ and $\rT_1(N,s;q)$ established by Andrews and Baxter~\cite{AB1987}.

\begin{lemma}
	For $s\in \mathbb{Z}$,
	\begin{align}\label{eq:T-rec-1}
		\Tfnc{1}{N-1}{s}{q} &= \Tfnc{1}{N-2}{s}{q} + q^{N+s-1} \Tfnc{0}{N-2}{s+1}{q}\notag\\
		&\quad + q^{N-s-1} \Tfnc{0}{N-2}{s-1}{q},
	\end{align}
	and
	\begin{align}\label{eq:T-rec-2}
		\Tfnc{1}{N-1}{s}{q} &= q^{-N-s} \Tfnc{0}{N}{s}{q} - q^{-N-s} \Tfnc{0}{N-1}{s-1}{q}\notag\\
		&\quad - q^{N+s} \Tfnc{0}{N-1}{s+1}{q}.
	\end{align}
\end{lemma}

\begin{proof}
	For \eqref{eq:T-rec-1}, see \cite[p.~300, eq.~(2.16)]{AB1987}; for \eqref{eq:T-rec-2}, see \cite[p.~301, eq.~(2.19)]{AB1987}.
\end{proof}

In \eqref{eq:T-rec-2}, we substitute $N$ by $N-1$, replace $s$ with $-s$, and invoke the symmetries
\begin{align*}
	\Tfnc{0}{N}{s}{q} = \Tfnc{0}{N}{-s}{q}, \qquad\qquad \Tfnc{1}{N}{s}{q} = \Tfnc{1}{N}{-s}{q}.
\end{align*}
Then
\begin{align}\label{eq:T-rec-2b}
	\Tfnc{1}{N-2}{s}{q} &= q^{-N+s+1} \Tfnc{0}{N-1}{s}{q} - q^{-N+s+1} \Tfnc{0}{N-2}{s+1}{q}\notag\\
	&\quad - q^{N-s-1} \Tfnc{0}{N-2}{s-1}{q}.
\end{align}
Finally, in \eqref{eq:T-rec-1}, we substitute the $\rT_1$ term on the left-hand side using \eqref{eq:T-rec-2}, and the $\rT_1$ term on the right-hand side using \eqref{eq:T-rec-2b}. The required relation \eqref{eq:rec-T0-to-show} immediately follows, thereby completing the proof of \eqref{eq:GF-2}.

\subsection{Kleshchev bipartition set $\Lambda^{1,2}$}

Note that
\begin{align*}
	G_N^{(1)}(x) = \sum_{s=-\infty}^\infty x^s \sum_{\substack{\pi \in \Lambda^{1,2}\\\lrg(\pi)\le N\\ \omega(\pi) = s}} q^{|\pi|}.
\end{align*}
In view of \eqref{eq:fnc-1}, we have
\begin{align*}
	\sum_{\substack{\pi \in \Lambda^{1,2}\\\lrg(\pi)\le N\\ \omega(\pi) = s}} q^{|\pi|} = q^{-N-1} (c_{s+\parity(N+1),N+1} - c_{s+\parity(N+1),N}),
\end{align*}
Since
\begin{align*}
	c_{s,N} = \sum_{\substack{\pi \in \Lambda^{2,2}\\\lrg(\pi)\le N\\ \omega(\pi) = s}} q^{|\pi|} = q^{s(s-1)}\Tfnc{0}{N}{s-N+2\lfloor\tfrac{N}{2}\rfloor}{q},
\end{align*}
we have
\begin{align}\label{eq:GF-1b} 
	\sum_{\substack{\pi \in \Lambda^{1,2}\\\lrg(\pi)\le N\\ \omega(\pi) = s}} q^{|\pi|} = q^{s^2+(-1)^N s-N-1} \left[\Tfnc{0}{N+1}{s}{q} - \Tfnc{0}{N}{s+(-1)^N}{q}\right].
\end{align}
While this relation is true, it is not well behaved, especially when we take the limit at $N\to \infty$. Now we derive the claimed identity \eqref{eq:GF-1} from \eqref{eq:GF-1b} to have a better expression for the desired counting function.

For $N=2n$, we have
\begin{align*}
	\RHS\eqref{eq:GF-1b} &= q^{s(s+1)-2n-1} \left[\Tfnc{0}{2n+1}{s}{q} - \Tfnc{0}{2n}{s+1}{q}\right]\\
	&= q^{s(s+1)-2n-1} \left[\Tfnc{0}{2n+1}{-s}{q} - \Tfnc{0}{2n}{-s-1}{q}\right].
\end{align*}
Applying \eqref{eq:T-rec-2},
\begin{align*}
	\RHS\eqref{eq:GF-1b} &= q^{s(s+1)-2n-1} \left[q^{2n-s+1}\Tfnc{1}{2n}{-s}{q} + q^{4n-2s+2} \Tfnc{0}{2n}{-s+1}{q}\right]\\
	&= q^{s^2}\left[\Tfnc{1}{2n}{s}{q} + q^{2n-s+1}\Tfnc{0}{2n}{s-1}{q}\right].
\end{align*}
For $N=2n+1$, we have, again with \eqref{eq:T-rec-2} used, that
\begin{align*}
	\RHS\eqref{eq:GF-1b} &= q^{s(s-1)-2n-2} \left[\Tfnc{0}{2n+2}{s}{q} - \Tfnc{0}{2n+1}{s-1}{q}\right]\\
	&= q^{s(s-1)-2n-2} \left[q^{2n+s+2}\Tfnc{1}{2n+1}{s}{q} + q^{4n+2s+4} \Tfnc{0}{2n+1}{s+1}{q}\right]\\
	&= q^{s^2}\left[\Tfnc{1}{2n+1}{s}{q} + q^{2n+s+2}\Tfnc{0}{2n+1}{s+1}{q}\right].
\end{align*}
Hence, \eqref{eq:GF-1} follows.

\section{Closing remarks}

Let $\dcD$ denote the set of strict partitions. Define the generating function
\begin{align*}
	\Sigma_N(x,q) := \sum_{\substack{\lambda\in \dcD\\ \lrg(\lambda) \le N}} x^{\omega(\lambda)} q^{|\lambda|}.
\end{align*}
Berkovich and Uncu~\cite[p.~13, Theorem~3.1]{BU2016} derived the following identity:

\begin{lemma}
	For $N\in \mathbb{Z}_{\ge 0}$,
	\begin{align}\label{eq:Sigma}
		\Sigma_N(x,q) = \sum_{i=-\infty}^\infty x^i q^{2i^2-i} \qbinom{N}{\lfloor\frac{N}{2}\rfloor+i}_{q^2}.
	\end{align}
\end{lemma}

We further introduce
\begin{align*}
	\Psi_{M,N}(x,q) := \sum_{\substack{\lambda\in \dcD\\ \len(\lambda) = M\\ \lrg(\lambda) \le N}} x^{\omega(\lambda)} q^{|\lambda|}.
\end{align*}
A consequence of \cite[p.~27, Proposition~7.2]{BU2016} is as follows:

\begin{lemma}\label{le:Psi}
	For $M,N\in \mathbb{Z}_{\ge 0}$ with $N\ge M$,
	\begin{align}\label{eq:Psi}
		\Psi_{M,N}(x,q) &= x^{(-1)^{M-1} \lceil\frac{M}{2}\rceil} q^{\binom{M+1}{2}} \sum_{j=0}^M (x^{(-1)^M} q)^j\notag\\
		&\quad\times \qbinom{\lceil\frac{N-M}{2}\rceil + j-1}{j}_{q^2} \qbinom{\lfloor\frac{N-M}{2}\rfloor + M-j}{M-j}_{q^2}.
	\end{align}
\end{lemma}

It is clear that
\begin{align*}
	\sum_{\substack{\pi \in \Lambda^{2,2}\\\lrg(\pi)\le N\\ \omega(\pi) = s}} q^{|\pi|} = \sum_{M=0}^N \Sigma_M(x,q) \Psi_{M,N}(x,q).
\end{align*}
Now splitting the sum according to the parity of $M$, and then invoking \eqref{eq:Sigma} and \eqref{eq:Psi}, our relation \eqref{eq:GF-2} implies two $q$-identities:

\begin{theorem}
	For $n\in \mathbb{Z}_{\ge 0}$ and $s\in \mathbb{Z}$,
	\begin{align}\label{eq:q-id-1}
		&q^{s(s-1)}\Tfnc{0}{2n}{s}{q}\notag\\
		&\qquad= \sum_{m=0}^n \sum_{j=0}^{2m} q^{P_s(j,2m)} \tqbinom{2m}{-s+j}_{q^2} \tqbinom{n-m+j-1}{j}_{q^2} \tqbinom{n+m-j}{2m-j}_{q^2}\notag\\
		&\qquad\quad + \sum_{m=0}^{n-1} \sum_{j=0}^{2m+1} q^{P_{1-s}(j,2m+1)} \tqbinom{2m+1}{s+j-1}_{q^2} \tqbinom{n-m+j-1}{j}_{q^2} \tqbinom{n+m-j}{2m-j+1}_{q^2},
	\end{align}
	and
	\begin{align}\label{eq:q-id-2}
		&q^{s(s-1)}\Tfnc{0}{2n+1}{s-1}{q}\notag\\
		&\qquad= \sum_{m=0}^n \sum_{j=0}^{2m} q^{P_s(j,2m)} \tqbinom{2m}{-s+j}_{q^2} \tqbinom{n-m+j}{j}_{q^2} \tqbinom{n+m-j}{2m-j}_{q^2}\notag\\
		&\qquad\quad + \sum_{m=0}^{n} \sum_{j=0}^{2m+1} q^{P_{1-s}(j,2m+1)} \tqbinom{2m+1}{s+j-1}_{q^2} \tqbinom{n-m+j-1}{j}_{q^2} \tqbinom{n+m-j+1}{2m-j+1}_{q^2},
	\end{align}
	where
	\begin{align*}
		P_s(a,b) := (s-a-1)^2 + (s-a+b)^2 + s-1.
	\end{align*}
\end{theorem}

\begin{question}
	Find a direct $q$-hypergeometric proof of \eqref{eq:q-id-1} and \eqref{eq:q-id-2}.
\end{question}

\subsection*{Acknowledgements}

This work was supported by the Austrian Science Fund (No.~10.55776/F1002). I started considering the finite Kleshchev bipartitions during my visit to Penn State University in April 2025, and the hospitality is well appreciated. I am grateful to Ae Ja Yee for pointing out Lemma~\ref{le:Psi} to me during this visit.

\bibliographystyle{amsplain}

\end{document}